\documentclass[hidelinks,onefignum,onetabnum]{siamart250106}

\ifpdf
\hypersetup{
  pdftitle={A geometric ensemble method for Bayesian inference},
  pdfauthor={A. A. Popov}
}
\fi

\usepackage{graphicx,wrapfig}

\usepackage{amsmath}

\newcommand{\appropto}{\mathrel{\vcenter{
  \offinterlineskip\halign{\hfil$##$\cr
    \propto\cr\noalign{\kern2pt}\sim\cr\noalign{\kern-2pt}}}}}

{%
\end{tcolorbox}\endgroup}

\usepackage{tabularx}

\usepackage{hyperref} 
\usepackage{graphicx} 

\usepackage{array}

\usepackage{amsmath,amssymb,amsfonts}
\usepackage{mathtools}
\usepackage{physics}

\usepackage[utf8]{inputenc}
\usepackage{newunicodechar,graphicx}
\DeclareRobustCommand{\okina}{%
  \raisebox{\dimexpr\fontcharht\font`A-\height}{%
    \scalebox{0.8}{`}%
  }%
}
\newunicodechar{ʻ}{\okina} 

\usepackage{lipsum}
\usepackage{amsfonts}
\usepackage{graphicx}
\usepackage{epstopdf}
\usepackage{algorithmic}
\ifpdf
  \DeclareGraphicsExtensions{.eps,.pdf,.png,.jpg}
\else
  \DeclareGraphicsExtensions{.eps}
\fi

\newsiamremark{remark}{Remark}
\newsiamremark{hypothesis}{Hypothesis}
\crefname{hypothesis}{Hypothesis}{Hypotheses}
\newsiamthm{claim}{Claim}
\newsiamremark{fact}{Fact}
\crefname{fact}{Fact}{Facts}

\headers{A geometric ensemble method}{A. A. Popov}

\title{A geometric ensemble method for Bayesian inference\thanks{Submitted to the editors DATE.
\funding{This work was supported by startup funding from the University of Hawai\okina{}i at M\=anoa }}}

\author{Andrey A. Popov\thanks{Department of Information and Computer Sciences, University of Hawai\okina{}i at M\=anoa, Honolulu, HI 
  (\email{apopov@hawaii.edu}, \url{https://www2.hawaii.edu/\string~apopov/}).}}

\usepackage{amsopn}

\DeclareMathOperator{\Vol}{Vol}

\newcommand{\mH}[0]{$\mathcal{H}$}

\usepackage{pgfplots}
\pgfplotsset{compat=1.18} 
\usepackage{pgfplotstable}
\usepackage{cbcolors}

\newcommand{\U}[0]{\mathcal{U}}

\begin{document}

\pgfplotsset{clean/.style={axis lines*=left,
        axis on top=true,
        axis x line shift=0.0em,
        axis y line shift=0.75em,
        every tick/.style={black, thick},
        axis line style = ultra thick,
        tick align=outside,
        clip=false,
        major tick length=4pt}}

\maketitle

\begin{abstract}
Conventional approximations to Bayesian inference rely on either approximations by statistics such as mean and covariance or by point particles.
Recent advances such as the ensemble Gaussian mixture filter have generalized these notions to sums of parameterized distributions.
This work presents a new methodology for approximating Bayesian inference by sums of uniform distributions on convex polytopes.
The methodology presented herein is developed from the simplest convex polytope filter that takes advantage of uniform prior and measurement uncertainty, to an operationally viable ensemble filter with Kalmanized approximations to updating convex polytopes.
Numerical results on the Ikeda map show the viability of this methodology in the low-dimensional setting, and numerical results on the Lorenz '96 equations similarly show viability in the high-dimensional setting.
\end{abstract}

\begin{keywords}
Bayesian inference, data assimilation, nonlinear estimation, convex polytopes, kernel density estimation
\end{keywords}

\begin{MSCcodes}
60G25, 62L12, 62M20, 93E11
\end{MSCcodes}

\section{Introduction}

Inference that does not incorporate all known, relevant, information is not Bayesian~\cite{jaynes2003probability}.
Limited computational resources prevent us from performing `true' Bayesian inference, and thus approximation that are non-Bayesian have to be made for tractability.
If the approximations have a notion of asymptotic convergence, then they can be said to converge to Bayesian inference. 
This work is concerned with convergent ensemble methods for Bayesian inference for data assimilation~\cite{reich2015probabilistic,asch2016data} and state estimation~\cite{bar2004estimation} applications. 
Specifically Monte Carlo-esque methods that in the limit of the number of samples converge in distribution~\cite{reich2015probabilistic} to `true' Bayesian inference.
Examples of convergent methods are the bootstrap particle filter (BPF)~\cite{stewart1992use,reich2015probabilistic}, the ensemble transform particle filter (ETPF)~\cite{reich2013nonparametric, acevedo2017second}, the ensemble Gaussian mixture filter (EnGMF)~\cite{anderson1999monte,liu2016efficient, yun2022kernel,popov2024adaptive}, and the point mass filter (PMF)~\cite{bergman1999terrain, giraldo2025gaussian}.
While this work focuses on data assimilation and state estimation, the concepts and methods developed herein are applicable to a wide range of Bayesian inference problems.
Some of the ideas presented in this work were indirectly inspired by the ideas presented in~\cite{nielen2024polytope}.

Recent advances~\cite{yun2022kernel,popov2024adaptive,popov2024Epanechnikov} to ensemble mixture model filters such as the EnGMF~\cite{anderson1999monte} have shown that convergent filters can be both computationally tractable and cheap for higher dimensional problems than traditional Monte Carlo methods.
Algorithms such as the ETPF have also shown that geometric constraint, such that posterior samples lie inside of the convex hull of the prior samples~\cite{van2015assimilating}, can be convergent in the correct circumstances, but require expensive, large matrix-valued linear programming solution and that do not scale to the higher-dimensional setting. 
Furthermore ideas of identifying probabilities in terms of bounded regions have been explored for the $H_\infty$ filter in~\cite{zasadzinski2003robust,chang2015new,liu2016polytopic}, and, most closely related to this work, for tracking applications~\cite{broman1986polytopes}.

Building off of these ideas, this work is concerned with building a new type of method for convergent Bayesian inference from an alternate geometric foundation based on convex polytopes~\cite{ziegler2012lectures}.
This paper is organized as follows: We first present existing background on polytopes in~\cref{sec:background}. 
Next, we present several new ideas including the convex polytope filter (CPF) that is the optimal filter for uniform distributions on convex polytopes, and `extended' approximations thereto in the form of the extended convec polytope filter (ECPF) for non-linear measurements  in section~\cref{sec:convex-polytope-filter}.
We then `Kalmanize' these filters to generalize them to non-uniform measurement distributions with the Kalmanized convex polytope filter (KCPF) and extended Kalmanized convex polytope filter (EKCPF) in~\cref{sec:Kalmanized-convex-polytope-filter}.
Subsequently, we present naive convergent ensemble-based filters, the bootstrap convex polytope filter (BCPF), the ensemble convex polytope filter (EnCPF) and the ensemble Kalmanize convex polytope filter (EnKCPF) in~\cref{sec:ensemble-polytope-filters}.
Penultimately, we provide sequential filtering numerical experiments on a low-dimensional problem (Ikeda map) and on a high-dimensional problem (Lorenz '96) in~\cref{sec:numerical-experiments}.
We conclude with some final remarks and a future outlook in~\cref{sec:conclusions}.

\section{Background}
\label{sec:background}

There are two different ways of expressing a convex polytope, through the convex hull of its vertices or through its bounding hyperplanes~\cite{ziegler2012lectures, henk2017basic}.
While these definitions are known to be equivalent, their computational properties are significantly different.
Take the $n$-dimensional unit cube and observe that it is bounded by $2n$ hyperplanes, which we later show are expressible using $\mathcal{O}(n^2)$ parameters, but it has $2^n$ vertices, which are expressible using $\mathcal{O}(n 2^n)$ parameters. 
Note that the opposite is true for an $n$-dimensional unit cross-polytope, as it is the dual of the unit cube. 
The vertex representation of a convex polytope is known as a $\mathcal{V}$-polytope and a (non-unique) bounding hyperplane representation is known as an \mH-polytope.
For ease of adding hyperplane constraints to a polytope, we opt to exclusively make use of the \mH-polytopes in this work.

We now formally define \mH-polyhedra and \mH-polytopes.
An \textit{\mH-polyhedron} is a subset of $\mathbb{R}^n$, where for the rest of this work $n$ is the dimension of interest, such that,
\begin{equation}
    Ph(A, b) = \{ x | A x \leq b,\,\, x\in\mathbb{R}^n\},
\end{equation}
where the matrix $A\in\mathbb{R}^{p\times n }$ and column vector $b\in\mathbb{R}^{p}$ define a collection of $p$ bounding hyperplanes.
If the set defined by $Ph(A,b)$ is bounded, then the convex polyhedron is an \textit{\mH-polytope},
\begin{equation}\label{eq:polytope}
    P(A, b) = \{ x | A x \leq b,\,\, x\in\mathbb{R}^n\},
\end{equation}
where the $p$ bounding hyperplanes are not necessarily all `active', meaning that there are infinitely many $A$ and $b$ that define the same set, and, that there is a lower bound for $p$, but there is no upper bound. 

For the rest of this work we will use the following shorthand,
\begin{equation}
    Ph_{\Box} \coloneqq Ph(A_\Box, b_\Box),\,\,\, P_{\Box} \coloneqq P(A_\Box, b_\Box),
\end{equation}
for expressing the hyperplanes defining the \mH-polyhedron $Ph_{\Box}$ and \mH-polytope $P_{\Box}$.

We now describe several useful properties of \mH-polytopes. Given two \mH-polytopes $P_1$ and $P_2$ then,
\begin{equation}
    P_1 \cap P_2 = \left\{x \,\,\,\middle| \,\,\,\begin{bmatrix}
        A_1\\A_2
    \end{bmatrix} x \leq \begin{bmatrix}
        b_1\\b_2
    \end{bmatrix}\right\},
\end{equation}
is also an \mH-polytope.
If $Ph_2$ is an \mH-polyhedron, then
\begin{equation}
    P_1 \cap Ph_2 = \left\{x \,\,\,\middle| \,\,\,\begin{bmatrix}
        A_1\\A_2
    \end{bmatrix} x \leq \begin{bmatrix}
        b_1\\b_2
    \end{bmatrix}\right\},
\end{equation}
is also an \mH-polytope.
If $C$ is an invertible matrix, and $d$ is some vector, then the affine transformation,
\begin{equation}
    C P_1 + d = P\left(A_1 C^{-1}, b_1 + A_1 C^{-1} d\right),
\end{equation}
defines an \mH-polytope.
It is important to note that a non-invertible transformation would define an \mH-polyhedron.
Finally, the Minkowski sum,
\begin{equation}\label{eq:Minkowski-sum}
    P_1 \oplus P_2 = \left\{x + y \,\,\,\middle| \,\,\, x\in P_1, y \in P_2\right\},
\end{equation}
defines an \mH-polytope.

\begin{remark}[\mH-polytope representation of the Minkowski sum]
    While it is true that \mH-polytopes are closed under Minkowski sum, computing the hyperplanes bounding the resulting polytope is a non-trivial affair, and is in practice intractable for higher dimensions~\cite{tiwary2007hardness}.
\end{remark}

We now define the uniform probability density function on a polytope $P$ as,
\begin{equation}
    \U(x\,;\, P) = \begin{cases}
        \frac{1}{\Vol(P)} & x \in P,\\
        0 & \text{sonst}
    \end{cases},\quad
    \Vol(P) = \int_P \mathrm{d}{x},
\end{equation}
where $\Vol(P)$ is the volume of the \mH-polytope $P$.

The Chebyshev center of an \mH-polytope $P$ is the center of the largest sphere that can be fully enclosed by said polytope.
The center $x_c$ and radius $r_c$ of this sphere are given by the solution to the following linear programming problem, 
\begin{equation}\label{eq:Chebyshev-center}
\begin{gathered}
    x_c,\,\, r_c = \operatorname{arg\,min}_{x, r} -r\\
    \text{s.t.}\quad \left(A x + \begin{bmatrix}\lVert a_1\rVert & \cdots & \lVert a_p\rVert\end{bmatrix}^T \right)r \leq b,
\end{gathered}
\end{equation}
where $\left\lVert a_i\right\rVert$ is the norm of the $i$th row of $A$.

The unit cube in $n$ dimensions is defined as,
\begin{equation}\label{eq:unit-cube}
    Q_n = P\left(
    \begin{bmatrix}
        I_n \\
        -I_n \\
    \end{bmatrix},
    \begin{bmatrix}
        0_n \\
        1_n \\
    \end{bmatrix}
    \right)
\end{equation}
where $I_n$ is the $n\times n$ identity matrix, $1_n$ is the $n$-vector of ones, and $0_n$ is the $n$-vector of zeros.

It is possible to modify a unit cube~\cref{eq:unit-cube} on which the uniform distribution would have a desired mean $\mu$ and covariance $\Sigma$ as follows:
\begin{equation}\label{eq:covariance-cube}
    Q_{\mu, \Sigma} = \sqrt{12} \Sigma^{1/2} \left(I_n Q_n - \frac{1}{2}1_n\right) + \mu,\quad \mathbb{E}[\U(Q_{\mu, \Sigma})] = \mu, \quad \mathbb{V}[\U(Q_{\mu, \Sigma})] = \Sigma,
\end{equation}
where $\sqrt{12}$ is a constant inherent to the uniform distribution, $\mathbb{E}$ is the expected value operator, and $\mathbb{V}$ is the covariance operator. The matrix square root that is used in this work is, 
\begin{equation}\label{eq:matrix-sqrt}
    \Sigma^{1/2} = V\sqrt{\Lambda}, \quad \Sigma = V\Lambda V^T,
\end{equation}
where the latter is the eigendecomposition of the covariance matrix into $V$ eigenvectors and a diagonal matrix $\Lambda$ of eigenvalues.

\begin{remark}
    Note that the storage cost of $Q_n$ in~\cref{eq:unit-cube} is $\mathcal{O}(n)$ as the identity matrices are sparse. 
    If a sparse decomposition to the square root of the covariance in~\cref{eq:covariance-cube} that also takes $\mathcal{O}(n)$ storage is used, then it is possible for $Q_{\mu, \Sigma}$ to only take $\mathcal{O}(n)$ storage instead of $\mathcal{O}(n^2)$ storage.
\end{remark}

Given a point $c$ on a hyperplane and a vector $v$ that is normal to the hyperplane,
\begin{equation}\label{eq:centroid-normal-hyperplane}
    a = v^T,\,\,\, b = a c,
\end{equation}
define a row vector $a$ and scalar $b$ that define the hyperplane such that all points $x$ for which $a x \leq b$ are opposite the direction of the normal vector.

We now turn our attention to a brief introduction to Bayesian inference. 
Assume that we have prior information about some quantity of interest in the form of the random variable $X^-$, whose probability distribution describes our uncertainty thereof (in  an alternate sense this would be called the hypothesis).
Additionally, assume that we have some other prior information $T$ whose probability distribution describes anything we know related to the quantity of interest that is external to its representation (in  an alternate sense this would be called the prior information).
We receive some additional information, related to the quantity of interest in the form of the random variable $Y$ whose probability distribution describes our uncertainty about it, such as about its veracity, reliability, or plausibility (in an alternate sense this would termed the data).
Bayesian inference combines these three sources of information as follows,
\begin{equation}\label{eq:Bayesian-inference}
    \mathbb{P}(X^+ | T) = \mathbb{P}(X^- | Y, T) = \frac{\mathbb{P}(Y | X^-, T)\mathbb{P}(X^- | T)}{\mathbb{P}(Y | T)},
\end{equation}
where $X^+$ is a random variable of our posterior information~\cite{jaynes2003probability}.
The goal of all the methods presented herein is to perform (approximations to) the inference in~\cref{eq:Bayesian-inference}.

\section{Convex Polytope Filter}
\label{sec:convex-polytope-filter}

We begin by reformulating classical linear filtering~\cite{jazwinski2007stochastic} ideas in terms of uniform distributions on \mH-polytopes.
Assume that we want to estimate the state $x_k$ at time index $k$ of some unspecified process. 
We have access to a measurement of $x_k$,
\begin{equation}
    y_k = H_k x_k + \eta_k
\end{equation}
where $H_k$ is some linear measurement operator, $\eta_k$ is a realization of noise fundamental to our understanding of the measurement uncertainty.
The random variable $Y_k$ represents our uncertainty about the measurement.
In the case where the measurement noise $\eta_k$ is represented by a uniform distribution on some \mH-polytope,
\begin{equation}\label{eq:measurement-noise}
    \eta_k \sim \U(P_{\eta_k}),
\end{equation}
the distribution that describes our measurement uncertainty about $y_k$ is simply,
\begin{equation}
    Y_k \sim \U(P_{Y_k}),
\end{equation}
where the underlying \mH-polytope, 
\begin{equation}
    P_{Y_k} = P_{\eta_k} + H_k x_k,
\end{equation}
is simply an affine transformation of the \mH-polytope of the measurement noise in~\cref{eq:measurement-noise}.

Given prior uncertainty about the state that is described by a uniform distribution on an \mH-polytope,
\begin{equation}\label{eq:CPF-prior-uncertainty}
    X^-_k \sim \U(P_{X^-_k}),
\end{equation}
where as in~\cref{eq:Bayesian-inference} the superscript $\Box^-$ represents prior information, with the superscript $\Box^+$ later representing posterior information.

Note that $P_{Y_k}$ is an \mH-polytope in the measurement space, and not in the space of the full system.
It is possible to transform the polytope from measurement to the full space to obtain the measurement \mH-polyhedra,
\begin{equation}\label{eq:CPF-Ph-Yk}
    Ph_{Y_k} = H_k^\dagger P_{Y_k} = H_k^\dagger P(A_{Y_k}, b_{Y_k}) =  Ph(A_{Y_k} H_k, b_{Y_k}),
\end{equation}
where $H_k^\dagger$ is some pseudo-inverse of $H_k$ that does not need to be explicitly computed.
As \mH-polytopes are closed under intersection with \mH-polyhedra, the posterior distribution can simply be written as,
\begin{equation}\label{eq:CPF}
\begin{gathered}
    p_{X^+_k}(x) = \U(x\,;\,P_{X^+_k})\\
    P_{X^+_k} = P_{X^-_k} \cap {Ph}_{Y_k},
\end{gathered}
\end{equation}
thus defining the action of the convex polytope filter (CPF). 
A similar formulation appears in the literature in~\cite{broman1986polytopes}, though the formulation presented in this work was independently derived.

\begin{remark}[Propagation in the CPF]
Linear propagation of the random variable
    \begin{equation}
        X_k^- = M_k X^+_{k-1} + \Xi_k,\quad p_{\Xi_k}(x) = \U(x\,;\,P_{\Xi_k}),
    \end{equation}
    where linear model $M_k$ has model error about which our uncertainty is described by the distribution of the random variable $\Xi_k$, then the propagated distribution is described by,
    \begin{equation}
        p_{X_k^-}(x) = \U(x\,;\,P_{X_k^-}), \quad P_{X_k^-} = M_k P_{X^+_{k-1}} \oplus P_{\Xi_k}
    \end{equation}
    where the posterior \mH-polytope is a linear transformation coupled together with a Minkowski sum~\cref{eq:Minkowski-sum} on the model error \mH-polytope.
    As this is highly impractical to compute in the general high-dimensional case, this formulation of propagation is largely of independent interest.
\end{remark}

\subsection{Extended Convex Polytope Filter}
We now turn our attention to non-linear measurement operators,
\begin{equation}\label{eq:non-linear-measurement}
    y_k = h(x_k) + \eta_k,
\end{equation}
where the measurement operator $h$.
For most interesting $h$, the true posterior distribution is still uniform (provided that $\eta_k$ is uniform), but is not defined on an \mH-polytope (see the example in the next section).
Our goal is to linearize $h$ in such a way as the posterior is an \mH-polytope, which we now provide.

\begin{lemma}[Measurement inverse approximation]\label{eq:measurement-inverse-approx}
    Assume that a right-inverse of $h_k(x)$ exists. 
    A first order affine approximation thereof is given by,
    \begin{equation}
        h_k^\dagger(z) \approx \mu^-_k - H_k^\dagger(\mu^-_k)\,h_k(\mu^-_k) + H_k^\dagger(\mu^-_k)\,z,
    \end{equation}
    where $\mu^-_k$ is the mean of the prior, and,
    \begin{equation}\label{eq:measurement-Jacobian}
        H_k(\mu^-_k) = \left.\frac{\mathrm{d} h}{\mathrm{d} x}\right|_{x = \mu^-_k},
    \end{equation}
    is the Jacobian of $h_k(x)$ evaluated at the mean.
\end{lemma}
\begin{proof}
    Observe that a first order affine expansion of $h_k(x)$ can be written as,
    \begin{equation}
        h_k(x) \approx h_k(\mu^-_k) +  H_k(\mu^-_k)( x - \mu^-_k),
    \end{equation}
    meaning that, by simple algebraic manipulation,
    \begin{equation}
        h_k\circ h_k^\dagger(z) \approx z,
    \end{equation}
    as required.
\end{proof}

Similar to~\cref{eq:CPF-Ph-Yk}, we can make use of~\cref{eq:measurement-inverse-approx} in order to approximate the measurement polyhedron.

\begin{theorem}[Extended Measurement Polyhedron Approximation]\label{thm:measurement-polyhedron-approximation}
    By simple application of~\cref{eq:measurement-inverse-approx} an approximation to the measurement \mH-polyhedron in full space is given by,
    \begin{equation}
        Ph_{Y_k} \approx H_k(\mu^-_k)^\dagger P_{Y_k} + \mu^-_k - H_k(\mu^-_k)^\dagger\, h(\mu^-_k),
    \end{equation}
    which can be written as
    \begin{equation}
    \begin{multlined}
        Ph_{Y_k}  = H_k(\mu^-_k)^\dagger P(A_{Y_k},\,\, b_{Y_k})  + \mu^-_k - H_k(\mu^-_k)^\dagger\, h(\mu^-_k) \\
        =  Ph(A_{Y_k} H_k(\mu^-_k),\,\, b_{Y_k} + A_{Y_k}H_k(\mu^-_k)\mu^-_k - A_{Y_k}h_k(\mu^-_k)),
        \end{multlined}
    \end{equation}
    without the use of the inverse of $H_k(\mu^-_k)$.
\end{theorem}

The posterior can be similarly defined as in~\cref{eq:CPF},
\begin{equation}\label{eq:ECPF}
\begin{gathered}
    p_{X^+_k}(x) = \U(x\,;\,P_{X^+_k}),\\
    P_{X^+_k} = P_{X^-_k} \cap {Ph}_{Y_k},
\end{gathered}
\end{equation}
which we term the extended convex polytope filter (ECPF).

\subsection{CPF and ECPF examples}

\begin{figure}
    \centering
    \includegraphics[width=0.49\linewidth]{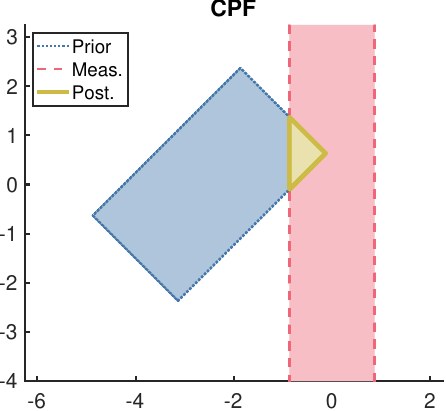}%
    \includegraphics[width=0.49\linewidth]{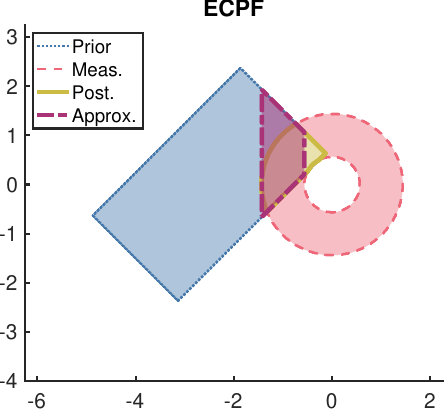}%
    \caption{Evaluation of the CPF (left panel) and the ECPF (right panel) on their respective test problems. On both panels, the left blue dotted rectangle represents the prior uniform distribution, the red dashed polyhedron (left panel) and doughnut (right panel) represent their respective measurement distributions, the solid yellow outline represents the exact posterior distribution, and the dash-dotted polytope (right panel) represents the ECPF approximation.}
    \label{fig:cpf-ecpf}
\end{figure}

We now present a simple numerical demonstration of the CPF and ECPF.
Take the 2-cube~\cref{eq:covariance-cube},
\begin{equation}\label{eq:CPF-example-prior}
    P_{X^-} = Q_{\mu^-, \Sigma^-},\quad \mu^- = \begin{bmatrix}
        -2.5\\ 0
    \end{bmatrix},\quad \Sigma^- = \begin{bmatrix}
        1 &0.5\\ 0.5 & 1
    \end{bmatrix},
\end{equation}
as the \mH-polytope that describes the prior uniform distribution.
Take the first measurement to be linear and measuring the first variable of $x$ and our uncertainty about the measurement described by a uniform distribution,
\begin{equation}
    H = \begin{bmatrix}
        1\\0
    \end{bmatrix},\quad p_{Y_k}(y) = \U(y\,;\,Q_{0, 1/4}),
\end{equation}
where $Q_{0, 1/4}$ is the interval $[-\sqrt{3}/2, \sqrt{3}/2]$. 
The above can be thought of as measuring the first state of $x$ and having uniform uncertainty of unit length centered at zero.

The \mH-polytope describing the uniform distribution of the posterior,
\begin{equation}
    P_{X^+} = Q_{\mu^-, \Sigma^-} \cap H^\dagger Q_{0, 1/4}
\end{equation}
can be exactly computed using the CPF equations~\cref{eq:CPF}. 
The left panel in~\cref{fig:cpf-ecpf} visually showcases this example.
Note that there is no approximation involved in this example as the CPF is exact for this problem setup.

For the ECPF we require a non-linear measurement operation, thus we take the range from the origin,
\begin{equation}
    h(x) = \lVert x \lVert_2, \,\, H(x) = \frac{x^T}{\lVert x \lVert_2},\quad p_{Y_k}(y) = \U(y\,;\,Q_{1, 1/16}),
\end{equation}
where $Q_{1, 1/16}$ is a doughnut in the state space. 
The \mH-polytope describing the uniform distribution of the posterior,
\begin{equation}
    P_{X^+} = Q_{\mu^-, \Sigma^-, n_\sigma} \cap \left[H(\mu^-)^\dagger Q_{1,1/16} - \mu^- + H_k(\mu^-)^\dagger\, h(\mu^-)\right]
\end{equation}
is approximated using the ECPF equations~\cref{eq:ECPF}. 
The right panel in~\cref{fig:cpf-ecpf} visually showcases this example.
Note that the approximated posterior \mH-polytope does not fully cover the true posterior, and that the approximated posterior has density where the true posterior does not, and vice versa.

\section{Kalmanized Convex Polytope Filter}
\label{sec:Kalmanized-convex-polytope-filter}

The CPF~\cref{eq:CPF} and ECPF~\cref{eq:ECPF} have two key limitations: i) that they require a uniform distribution on the measurement errors, and ii) that all posterior distributions have supports that are convex subsets of their respective priors.
Limitation i) is one of particular importance as many measurement errors are given in terms of standard deviations and not in terms of confidence intervals meaning that our knowledge about the measurement is not uniform.
Limitation ii) on the other hand is only a limitation when we are attempting to approximate Bayesian inference by deliberately simplifying our knowledge about the prior, which in general is the case in almost all practical applications and is thus of particular relevance to the practitioner.
In order to alleviate these problems we make use of the Kalman filter~\cite{kalman1960new} to `Kalmanize' the CPF and ECPF.

We now briefly describe the ideas underlying the Kalman filter.
Let all our information about the prior be expressed by its mean and covariance,
\begin{equation}\label{eq:prior-mean-covariance}
    \mathbb{E}[X^-_k] = \mu^-_k,\quad \mathbb{V}[X^-_k] = \Sigma^-_k,
\end{equation}
and let our uncertainty about the measurement be centered (have mean) on the measurement and have a known covariance,
\begin{equation}\label{eq:measurement-mean-covariance}
    \mathbb{E}[Y_k] = y_k,\quad \mathbb{V}[Y_k] = R_k,
\end{equation}
with no other information about our uncertainty known.

The Kalman filter equations are given by,
\begin{equation}\label{eq:Kalman-filter}
\begin{aligned}
    \mu^+_k &= \mu^-_k - K_k(H_k \mu^-_k - y_k),\\
    \Sigma^+_k &= (I_n - K_k H_k)\Sigma^-_k,\\
    K_k &= \Sigma^-_k H_k^T{(H_k \Sigma^-_k H_k^T + R_k)}^{-1},
\end{aligned}
\end{equation}
where the matrix $K$ is known as the Kalman gain, describe the update of the mean and covariance from the prior mean and covariance ($\mu^-_k$, $\Sigma^-_k$) to the posterior mean and covariance ($\mu^+_k$, $\Sigma^+_k$).

The Kalman filter is the best linear unbiased estimator of the given information~\cite{asch2016data,reich2015probabilistic}.
While methods such as the principle of maximum entropy~\cite{jaynes2003probability} would claim that the information in~\cref{eq:prior-mean-covariance} and~\cref{eq:measurement-mean-covariance} is sufficient to posit a Gaussian assumption, it is entirely unnecessary and carries with it extra `naturalness' assumptions.
This means that, most importantly, a frequent misconception about the Kalman filter is that it requires any type of Gaussian assumption, as it does not~\cite{uhlmann2024gaussianity,humpherys2012fresh}.
The Kalman filter is the best possible linear filter applied to all problems where only means and covariances are available to describe our knowledge irrespective of the underlying distributions involved.
This has also been shown through the lens of linear control variate theory in~\cite{popov2021multifidelity}.

Therefore we can comfortably \textit{Kalmanize} the CPF~\cref{eq:CPF} provided that we can know the mean and covariance of the uniform distribution on the prior~\cref{eq:CPF-prior-uncertainty}, and that we make the \textit{false but useful} assumption that that is all the information that we know.
In order to perform a Kalman-esque transform on the prior we need to be able to transform the underlying polytope $P_{X^-_k}$ through the use of the Kalman equations~\cref{eq:Kalman-filter}.
This is not straightforward as there are two different transformation formulas for the mean and covariance available in~\cref{eq:Kalman-filter}.
Fortunately we can take advantage of formulas derived for an ensemble Kalman square root filter~\cite{tippett2003ensemble,whitaker2002ensemble}, which present a modified Kalman gain,
\begin{equation}\label{eq:modified-Kalman-gain}
    \widetilde{K}_k = \Sigma^-_k H_k^T{\left(H_k \Sigma^-_k H_k^T + R\left(I_m + \sqrt{I_m + R^{-1}H_k \Sigma^-_k H_k^T}\right) \right)}^{-1}
\end{equation}
through which we can update \textit{anomalies} of the filter, $A^-_k = X^-_k - \mu^-_k$, as, 
\begin{equation}
    A^+ = (I_n - \widetilde{K}_k H_k) A^-_k.
\end{equation}
We are now ready to combine all these results into an affine update on the prior \mH-polytope.

\begin{theorem}[Kalmanized convex polytope filter]\label{thm:KCPF}
    The affine transformation that defines the Kalmanized convex polytope filter can  be written as,
    \begin{equation}\label{eq:KCPF}
    \begin{gathered}
        p_{X^+_k}(x) = \U_{P_{X^+_k}}(x),\\
        P_{X^+_k} = (I_n - \widetilde{K}_k H_k ) P_{X^-_k} + (\widetilde{K}_k - K_k)H_k\mu^-_k + K_k y_k,
    \end{gathered}
    \end{equation}
    where $p_{X^+_k}(x)$ is the distribution of the posterior and $P_{X^+_k}$ is an affine transformation of the prior \mH-polytope $P_{X^-_k}$.
\end{theorem}
\begin{proof}
    It suffices to show show the affine transformation~\cref{eq:KCPF} affects $X^-_k$.
    Observe that it is possible to combine the mean and anomaly updates as follows,
    \begin{equation*}
        X^+_k = \underbrace{\mu^-_k - K_k(H_k \mu^-_k - y_k)}_{\mu^+_k} +  \underbrace{(I_n - \widetilde{K}_k H_k) A^-_k}_{A^+_k},
    \end{equation*}
    which can be written as the following affine transformation on $X^-_k$,
    \begin{equation*}
        X^+_k = (I_n - \widetilde{K}_k H_k) X^-_k + (\widetilde{K}_k - K_k)H_k\mu^-_k + K_k y_k,
    \end{equation*}
    as required.
\end{proof}
We term~\cref{eq:KCPF} the Kalmanized convex polytope filter (KCPF).

Observe that for a non-linear measurement~\cref{eq:non-linear-measurement} it is possible to write the mean and anomaly update as follows,
\begin{equation}
\begin{aligned}
    \mu^+_k &= \mu^-_k - K_k(h_k(\mu^-_k) + H_k(\mu^-_k) \mu^-_k - H_k(\mu^-_k) \mu^-_k - y_k),\\
    A^+_k &= (I_n - \widetilde{K}_k H_k(\mu^-_k)) A^-_k,
\end{aligned}
\end{equation}
where $H_k(\mu^-_k)$ is the measurement Jacobian~\cref{eq:measurement-Jacobian} evaluated at the prior mean, $K_k$ is the Kalman gain from~\cref{eq:Kalman-filter} evaluated using the measurement Jacobian, and  $\widetilde{K}_k$ is the modified Kalman gain from~\cref{eq:modified-Kalman-gain}.
We are again ready to combine all these results into an affine update on the prior \mH-polytope for non-linear measurements.

\begin{corollary}
    Without proof, similar to~\cref{thm:KCPF}, the extended Kalmanized convex polytope filter can be written as,
    \begin{equation}\label{eq:EKCPF}
    \begin{gathered}
        p_{X^+_k}(x) = \U(x\,;\,P_{X^+_k}),\\
        P_{X^+_k} = (I_n - \widetilde{K}_k H_k(\mu^-_k) ) P_{X^-_k} + \widetilde{K}_k H_k(\mu^-_k)\mu^-_k + K_k y_k - K_k h_k(\mu^-_k),
    \end{gathered}
    \end{equation}
    where again $p_{X^+_k}(x)$ is the distribution of the posterior and $P_{X^+_k}$ is an affine transformation of the prior \mH-polytope $P_{X^-_k}$.
\end{corollary}
We term~\cref{eq:EKCPF} the extended Kalmanized convex polytope filter (EKCPF).

\subsection{KCPF and EKCPF examples}
\begin{figure}
    \centering
    \includegraphics[width=0.49\linewidth]{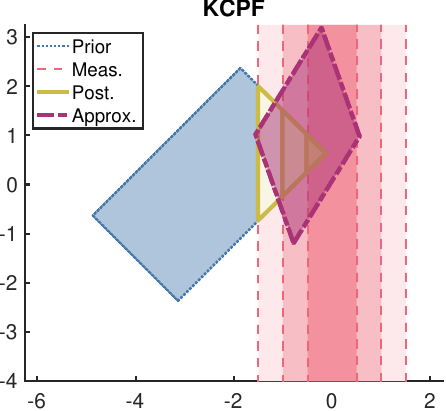}%
    \includegraphics[width=0.49\linewidth]{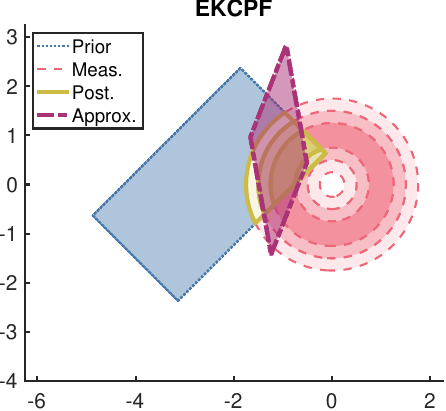}%
    \caption{Evaluation of the KCPF (left panel) and the EKCPF (right panel) on their respective test problems. On both panels, the left blue dotted rectangle represents the prior uniform distribution, the red dashed outlines represent their respective measurement distributions with each line representing one, two, and three standard deviations, the solid yellow outline represents the exact posterior distribution again with one, two, and three standard deviations, and the dash-dotted polytopes represents the KCPF and EKCPF approximations in the left and right panels respectively.}
    \label{fig:kcpf-ekcpf}
\end{figure}

We now present a simple numerical demonstration of the KCPF and EKCPF.
As in the previous example, take \cref{eq:CPF-example-prior} as the \mH-polytope that describes the prior uniform distribution.
Take the first measurement to be linear and measuring the first variable of $x$ and our uncertainty about the measurement described by a normal distribution,
\begin{equation}
    H = \begin{bmatrix}
        1\\0
    \end{bmatrix},\quad p_{Y_k}(y) = \mathcal{N}(y; 0, 1/4),
\end{equation}
where the normal distribution has mean zero and variance of $1/4$.
The above can be thought of as measuring the first state of $x$ and having normal uncertainty centered at zero.

The \mH-polytope approximating the uniform distribution of the posterior can be approximated using the KCPF equations~\cref{eq:KCPF}. 
The left panel in~\cref{fig:kcpf-ekcpf} visually showcases this example.
Note that the posterior \mH-polytope is a relatively poor approximation to the optimal (in KL-sense) uniform approximation to the posterior.

For the EKCPF we require a non-linear measurement operation we again take the range from the origin,
\begin{equation}\label{eq:EKCPF-example-measurement}
    h(x) = \lVert x \lVert_2, \,\, H(x) = \frac{x^T}{\lVert x \lVert_2},\quad p_{Y_k}(y) = \mathcal{N}(y; 1, 1/16),
\end{equation}
where this time we have a measurement of one with normally distributed measurement error with covariance $1/16$. 
The \mH-polytope describing the uniform distribution of the posterior, is approximated using the EKCPF equations~\cref{eq:EKCPF}. 
The right panel in~\cref{fig:cpf-ecpf} visually showcases this example.
Note again that the approximated posterior \mH-polytope does not fully cover the true posterior, but does cover most of one standard deviation.



\section{Ensemble Polytope Filters}
\label{sec:ensemble-polytope-filters}

Instead of being given a prior polytope like in~\cref{sec:convex-polytope-filter} or mean and covariance information like in~\cref{sec:Kalmanized-convex-polytope-filter}, we instead assume that the prior information that we have comes from some underlying distribution $p_{X^-_k}(x)$ from which we are given a finite collection of $N$ samples,
\begin{equation}\label{eq:ensemble}
       \underbrace{x_k^{-,(1)},\,\, x_k^{-,(2)},\,\, \dots,\,\, x_k^{-,(N)}}_{\sim p_{X^-_k}(x)},
\end{equation}
using which our goal is to update the ensemble by making use of measurement information.

The literature is filled with methods to perform such operations in an approximate manner.
Few of these methods, however converge (in the limit of ensemble size $N$, in distribution) to `true' Bayesian inference assuming all relevant information information is accounted for in the prior.


One convergent approach to convergent sequential filtering with ensembles is the bootstrap particle filter~\cite{stewart1992use, reich2015probabilistic}, which reweighs then resamples from the prior ensemble~\cref{eq:ensemble}.
This is done by making an implicit assumption that the underlying probability distribution of~\cref{eq:ensemble} is a sum of $\delta$-distributions on the prior ensemble and the often wrong assumption that there is no other possible external information about it.

Another approach to convergent sequential filtering is the ensemble Gaussian mixture filter (EnGMF)~\cite{anderson1999monte, liu2016efficient,yun2022kernel,popov2024adaptive}.
The EnGMF acts in three phases: i) Gaussian kernel density estimation~\cite{silverman2018density}, ii) the Gaussian sum update~\cite{sorenson1971recursive,anderson1999monte}, and iii) resampling from a posterior Gaussian mixture~\cite{popov2024adaptive}.
In a way this procedure can be thought of as a discrete-to-continuous approximation, manipulation of the continuous distribution, followed by a continuous-to-discrete transformation.
In particular the kernel density estimation step assumes that the prior ensemble~\cref{eq:ensemble} offers a limited view of the underlying prior probability distribution, and that it is possible for us to build a better estimate than that of a simple $\delta$-distribution sum approximation.
A more general description of ensemble mixture model filters on which this work is based can be found in~\cite{popov2024Epanechnikov}.



In the following sections our goal is first to i) build a general approximation of the prior with an \mH-polytope mixture from the ensemble~\cref{eq:ensemble},
\begin{equation}\label{eq:prior-polytope-mixture}
    p_{X^-_k}(x) \approx \sum_{i = 1}^N w^{X^-_k}_i\,\mathcal{U}\left(x\,;\,P_i^{X^-_k}\right),
\end{equation}
where $w^{X^-_k}_i$ are $N$ weights that sum to unity and represent our plausibility about any given mixture term.
Then ii) perform some update step to create a posterior polytope mixture model,
\begin{equation}\label{eq:posterior-polytope-mixture}
    p_{X^+_k}(x) \approx \sum_{i=1}^L w^{X^+_k}_i\,\U\left(x\,;\,P_i^{X^+_k}\right),
\end{equation}
where $w_i^{X^+_k}$ are weights that sum to one and each $P_i^{X^+_k}$ is an \mH-polytope.
Finally, iii) resample from~\cref{eq:posterior-polytope-mixture} to create new particles.
We first present a naive way to perform step i), then a way to perform step iii), and finally three different ways to perform step ii), as all the filters that we present only differ therein.

\subsection{Kernel density estimation using \texorpdfstring{\mH}{H}-polytopes}
\label{sec:En-CDE}

Instead of dealing with the problem of general density estimation using \mH-polytopes, in this work we restrict ourselves to kernel density estimation using \mH-polytopes with known mean and covariance, namely the mean shifted and covariance scaled cubes defined by~\cref{eq:covariance-cube} in order to create a mixture of the form~\cref{eq:prior-polytope-mixture}.
Similar to the EnGMF that takes advantage of Gaussian mixtures centered at each ensemble member~\cref{eq:ensemble} with a common covariance, we center uniform distributions on the scaled $n$-cube~\cref{eq:covariance-cube} with a common covariance, begetting the mixture model,
\begin{equation}\label{eq:prior-cubes}
    p_{X^-_k}(x) \approx \frac{1}{N}\sum_{i = 1}^N \mathcal{U}\left(x\,;\,Q_{x_{k}^{-,(i)}, h_\U^2 \Sigma^-_k}\right),
\end{equation}
where each uniform distribution is centered at each particle $x_{k}^{-,(i)}$ with covariance $h_\U^2 \Sigma^-_k$ where $\Sigma^-_k$ is the global covariance of the system (or, if not available, an estimate of the covariance of the ensemble~\cref{eq:ensemble}), and $h_\U$ is a bandwidth. 
We now provide one type of optimal bandwidth.

\begin{theorem}
Assuming that the underlying samples~\cref{eq:ensemble} come from a Gaussian distribution, the optimal bandwidth in terms of asymptotic mean integral squared error (AMISE) is given by,
\begin{equation}
    h_{\U} = {\left[ \frac{4\sqrt{\frac{\pi}{3}}^n n}{(n + 2)N} \right]}^{\frac{1}{n+4}}
\end{equation}
which is known as the Silverman~\cite{silverman2018density} bandwidth.
\end{theorem}
\begin{proof}
    For a kernel with mean zero and unit covariance, by the derivation in~\cite{silverman2018density} and exposition in~\cite{popov2024Epanechnikov}, the optimal bandwidth is given by
    \begin{equation}
       h_{\U} = \left[\frac{\beta n}{\gamma N}\right]^{\frac{1}{n + 4}},
    \end{equation}
    where the first parameter, $\beta$, is given by,
    \begin{equation}
        \beta = \int_{\mathbb{R}^n} \mathcal{K}(x)^2 \mathrm{d} x,\\
    \end{equation}
    with, without loss of generality (by simple affine transformation), $\mathcal{K}$ being the underlying kernel with mean zero and unit covariance.
    The parameter $\gamma$ depends on the distribution from which the samples are drawn.
    As we assume that the underlying samples are drawn from a Gaussian, the parameter $\gamma$ is given by,
    \begin{equation}
        \gamma = \frac{1}{2^n \sqrt{\pi}^n}\left(\frac{1}{2} n + \frac{1}{4}n^2\right).
    \end{equation}
    Observe further that as the kernel is uniform over the unit $n$-cube, that,
    \begin{equation}
        \beta = \int_{Q_{0,I}} (\U_{Q_{0,I}}(x))^2 \mathrm{d}x = \Vol(Q_{0,I})^{-1} = \sqrt{12}^{-n}, 
    \end{equation}
    and by simple algebraic manipulation the proof is complete.
\end{proof}

We now present a result about the convergence of this prior mixture to the true distribution that follows the theory presented in~\cite{popov2024adaptive}.
\begin{theorem}\label{thm:asymptotic-prior-cube-convergence}
    In the asymptotic limit of ensemble size $N\to\infty$, the prior mixture model~\cref{eq:prior-cubes} converges to the true distribution of the prior of the ensemble~\cref{eq:ensemble}.
\end{theorem}
\begin{proof}
Observe that the Silverman bandwidth for a uniform distribution tends to zero in the limit of ensemble size,
\begin{equation*}
     \lim_{N\to\infty} h_{\U} = 0,
\end{equation*}
meaning that the prior mixture model~\cref{eq:prior-cubes} tends towards a sum of $\delta$ distributions, which tends towards the `true' prior distribution, 
\begin{equation*}
    \lim_{N\to\infty} \frac{1}{N}\sum_{i = 1}^N \mathcal{U}\left(x\,;\,Q_{x_{k}^{-,(i)}, h_\U^2 \Sigma^-_k}\right) = \lim_{N\to\infty}  \frac{1}{N}\sum_{i = 1}^N \delta\left(x\,;\,x_{k}^{-,(i)}\right) = p_{X^-_k}(x),
\end{equation*}
as required.
\end{proof}

We now present a simple quadrature rule that can be used to approximate integrals over some $n$-cube $\U(Q_{\mu,\Sigma})$ from~\cref{eq:covariance-cube}.
For the unit $n$-cube~\cref{eq:covariance-cube}, $Q_{\mu,\Sigma}$ the collection containing the center of the cube and the centers of the facets are given by,
\begin{equation}\label{eq:Omega-points}
    \Omega_{\mu,\Sigma} = \sqrt{3}\begin{bmatrix}
       -V\sqrt{\Lambda} & 0_n  & V\sqrt{\Lambda}
    \end{bmatrix} + \mu\, 1_n^T,
\end{equation}
where $V$ and $\Lambda$ are the matrices of eigenvectors and eigenvalues defined by~\cref{eq:matrix-sqrt}, and where we shift the matrix indices to say that the first $n$ columns are columns $-n$ to $-1$, the middle column is column $0$ and the last $n$ columns are columns $1$ to $n$.
Observe that the columns of~\cref{eq:Omega-points}, excluding the $0$th column, represent the centers of the facets of $Q_{\mu,\Sigma}$,
By using each center of the facet as centroid of the hyperplane and the center of the cube to define a normal vector and by using~\cref{eq:centroid-normal-hyperplane} the points $\Omega_{\mu,\Sigma}$ are capable of fully representing the $n$-cube $Q_{\mu,\Sigma}$, and thus no loss of information should occur in this alternate representation, meaning that it should be possible to recover the underlying mean and covariance of $Q_{\mu,\Sigma}$ using only~\cref{eq:Omega-points}.
This is similar to the intuition behind the `$\sigma$-points' in the unscented (Uhlmann) transform~\cite{uhlmann1995dynamic,sarkka2023bayesian}, thus we similarly the name the points~\cref{eq:Omega-points} as $\omega$-points.
We now derive the weights with which these $\omega$-points~\cref{eq:Omega-points} can be used as quadrature points to approximate integrals of interest.
\begin{theorem}[$\omega$-point quadrature rule]\label{thm:Omega-point-quadrature-rule}
    The quadrature rule on the $\omega$-points,
    \begin{equation}\label{eq:Omega-point-quadrature-rule}
        \int_{\mathbb{R}^n} f(x)\, \U(x\,;\,Q_{\mu,\Sigma}) \mathrm{d} x \approx \sum_{\ell=-n}^n W^\Omega_{\ell}\,f(\Omega_{\ell,\mu,\Sigma})
    \end{equation}
    preserves the mean and covariance when the weights are defined as,
    \begin{equation}
        W^\Omega_{0} = 1,\quad \forall \ell \in [1, \dots, n]\,, W^\Omega_{-\ell} = W^\Omega_{\ell} = \frac{n}{3}.
    \end{equation}
\end{theorem}
\begin{proof}
    Observe trivially that the condition,
    \begin{equation*}
        W^\Omega_0 = 1,\quad \forall \ell \in [1, \dots, n]\,, W^\Omega_{-\ell} = W^\Omega_{\ell},
    \end{equation*}
    ensures the preservation of the mean. 
    We can further make the assumption that all scaled eigenvectors are of equal importance, thus all that is left to define is a single off-center weight. 
    Observe further that the estimate of the covariance by~\cref{eq:Omega-point-quadrature-rule} is given by,
    \begin{equation*}
        \mathbb{V}[\U(x\,;\,Q_{\mu,\Sigma})] \approx \frac{1}{2n} \sum_{\ell=1}^n 6\, W^\Omega_{\ell}\, V_\ell \Lambda_\ell V_\ell^T,
    \end{equation*}
    which is equal to $\Sigma$ when $W^\Omega_\ell = \frac{n}{3}$ for all $\ell \in \pm[1, \dots, n]$, as required.
\end{proof}

\begin{remark}[General \mH-polytope density estimation]
    The method presented in this section is simple in form, but is certainly not the optimal way of approximating densities through \mH-polytopes.
    A more robust method with unique \mH-polytopes per ensemble member, similar to that for the EnGMF developed in~\cite{popov2025ensemble}, is of future interest.
\end{remark}

\subsection{Resampling from a uniform \texorpdfstring{\mH}{H}-polytope mixture}
\label{sec:polytope-resample}

We now provide a method to resample from the posterior uniform \mH-polytope mixture~\cref{eq:posterior-polytope-mixture}.
We initially sample an \mH-polytope from the discrete distribution defined by the posterior weights $w^{X^+_k}_{i}$.
For the sake of exposition we assume that the result of this procedure is some \mH-polytope $P_\ell = P(A_\ell, b_\ell)$.

We now present a simple algorithm to sample from a uniform distribution on an \mH-polytope that is based on the `hit-and-run' algorithm~\cite{smith1984efficient}.
The start of the algorithm requires a point inside of the polytope, which, for a generic polytope can be the Chebyshev center~\cref{eq:Chebyshev-center}.
Call this starting point $z_0$.

For some generic step $s < S$, first we uniformly sample a direction from the unit sphere embedded in $n$ dimensions,
\begin{equation}
    u_s \sim \mathcal{U}(\mathcal{S}^{n-1}).
\end{equation}
The next step is the `hit', where we extend from $z_s$ into the direction $u_s$ both backwards and forwards to compute the maximum length in either direction,
\begin{equation}
    r_{s, \pm} = \operatorname{arg\,min}_r \mp r, \quad \text{s.t.} \,\,\, r A_\ell u_s \leq b_\ell - A z_s
\end{equation}
which has a trivial closed form solution that does not involve solving a linear programming problem.
A point is sampled from the uniform distribution between this minimum and maximum,
\begin{equation}
    r_s \sim \mathcal{U}([r_{s,-}, r_{s,+}])
\end{equation}
and, finally we `run' in the direction $u_s$ with a length of $r_s$.
\begin{equation}
    z_{s+1} = z_s + r_s u_s.
\end{equation}
which we repeat until we obtain $z_S$, which we take to be our sample. 
As $S\to\infty$ the sample becomes a sample from the target distribution.

For all ensemble filters subsequently presented, given $N$ prior samples the resampling procedure in this section is used to produce $N$ samples from the posterior, as methods with varying ensemble size $N$ are outside the scope of this work.

\subsection{Bootstrap convex polytope filter}
\label{sec:BCPF-update}

We now provide the \mH-polytope analogue to the bootstrap particle filter~\cite{stewart1992use, reich2015probabilistic}.
Starting from the general \mH-polytope mixture model in~\cref{eq:prior-polytope-mixture}, assume that we have uncertainty from a measurement given by the general distribution,
\begin{equation}
    p_{Y_k}(y\,;\, y_k),
\end{equation}
where the uncertainty is centered on the measurement realization $y_k$.
One naive way in which the posterior polytope mixture can be written is,
\begin{equation}\label{eq:BCPF-posterior}
\begin{gathered}
    p_{X^+_k}(x) \approx \sum_{i = 1}^N w^{X^+_k}_{i} \mathcal{U}\left(x\,;\, P_{i}^{X^-_k}\right),\\
    w^{X^+_k}_{i} \propto w^{X^-_k}_{i}\int_{\mathbb{R}^n} \mathcal{U}\left(x\,;\,P_i^{X^-_k}\right) p_{Y_k}(h_k(x)\,;\, y_k)\mathrm{d}x,
\end{gathered}
\end{equation}
where the weights can be approximated through the quadrature rule deifned in~\cref{thm:Omega-point-quadrature-rule}.
The process of approximating the prior using a method such as in~\cref{sec:En-CDE}, updating the weights, and resampling from the posterior~\cref{eq:BCPF-posterior} as defined in~\cref{sec:polytope-resample} we term the bootstrap convex polytope filter (BCPF).
Under certain assumptions, this procedure converges to true Bayesian inference in distribution.
\begin{corollary}[BCPF convergence]\label{thm:BCPF-convergence}
    Given the prior polytope mixture as~\cref{eq:prior-cubes}, then the posterior defined by~\cref{eq:BCPF-posterior} converges to the `true' posterior by a trivial extension of~\cref{thm:asymptotic-prior-cube-convergence}.
\end{corollary}

For the BCPF, instead of using the Chebyshev center, it is computationally cheaper to start the resampling procedure in~\cref{sec:polytope-resample} from the posterior mean of each component, which is identical to the prior mean of each component.

\subsection{Ensemble convex polytope filter}
\label{sec:EnCPF-update}

We now provide the \mH-polytope analogue to the ensemble Gaussian mixture filter~\cite{anderson1999monte, yun2022kernel, liu2016efficient,popov2024adaptive}.
Starting from the general \mH-polytope mixture model in~\cref{eq:prior-polytope-mixture}, assume that we have uncertainty from a measurement given by the uniformly distributed \mH-polytope mixture,
\begin{equation}\label{eq:EnCPF-measurement-distribution}
    p_{Y_k}(y) = \sum_{j=1}^M w^{Y_k}_j\,\mathcal{U}(y\,;\,P_j^{Y_k}),
\end{equation}
where $w^{Y_k}_j$ are $M$ weights that sum to unity, then by~\cref{thm:measurement-polyhedron-approximation} the collection of $N\times M$ posterior \mH-polytopes is given by,
\begin{equation}
    P_{i,j}^{X^+_k} = P_i^{X^-_k} \cap \left[H_k(\mu^-_{i,k})^\dagger P_j^{Y_k}  + \mu^-_{i,k} - H_k(\mu^-_{i,k})^\dagger\, h_k(\mu^-_{i,k})\right],\\
\end{equation}
where $\mu_{i,k}^-$ is the mean of $\mathcal{U}(x\,;\,P_i^{X^-_k})$.

The resulting approximation to the posterior is given by,
\begin{equation}\label{eq:EnCPF-posterior}
    p_{X^+_k}(x) \approx \sum_{i = 1}^N\sum_{j=1}^M w^{X^+_k}_{i,j} \mathcal{U}(x\,;\, P_{i,j}^{X^+_k}),
\end{equation}
where the weights are given by,
\begin{equation}
    w^{X^+_k}_{i,j} \propto w^{X^-_k}_i\, \int_{\mathbb{R}^n} \mathcal{U}(x\,;\,P_i^{X^-_k}) \, \mathcal{U}(h_k(x)\,;\,P_j^{Y_k}) \mathrm{d}x,
\end{equation}
which, for non-linear measurements~\cref{eq:non-linear-measurement} are non-trivial to compute exactly, and are non-trivial to approximate, even with the $\omega$-points from~\cref{thm:Omega-point-quadrature-rule}, and is therefore only provided herein for completeness.

The process of approximating the prior using a method such as in~\cref{sec:En-CDE}, and resampling from the posterior as defined in~\cref{sec:polytope-resample} we term the bootstrap convex polytope filter (BCPF).

\begin{corollary}[EnCPF convergence]
    Observe that as the prior estimate approaches a sum of $\delta$ distributions, the posterior in~\cref{eq:EnCPF-posterior} approaches that of the BCPF for the particular measurement uncertainty distribution~\cref{eq:EnCPF-measurement-distribution}.
    Trivially, this implies that the EnCPF converges to the BCPF which converges to the BPF which converges in distribution to true Bayesian inference in the limit of ensemble size $N$.
\end{corollary}

\subsection{Ensemble Kalmanized convex polytope filter}
\label{sec:EnKCPF-update}

We now provide a hybrid Kalmanized \mH-polytope analogue of the ensemble Gaussian mixture filter~\cite{anderson1999monte, yun2022kernel, liu2016efficient,popov2024adaptive}.
Similar to the EnGMF and other mixture model filters~\cite{popov2024Epanechnikov}, we can make use of a collection of Kalman-type updates.
This is frequently (though incorrectly) labeled as Gaussian sum update~\cite{popov2024adaptive,popov2024Epanechnikov,sorenson1971recursive}.

Starting from the general \mH-polytope mixture model in~\cref{eq:prior-polytope-mixture}, 
and take $\mu_i^{X^-_k}$ to be the mean of the $i$th component of~\cref{eq:prior-polytope-mixture}, and $\Sigma_i^{X^-_k}$ to be its covariance.
Assume that we have uncertainty from a measurement given by the more general mixture model,
\begin{equation}
    p_{Y_k}(y) = \sum_{j=1}^M w^{Y_k}_j\,p_j(y\,;\,y_j^{Y_k},\, R_j^{Y_k}),
\end{equation}
where $w^{Y_k}_j$ are $M$ weights that sum to unity, $y_j^{Y_k}$ is the mean of the $j$th mixture, and $R_j^{Y_k}$ is the covariance of the $j$th mixture.
In general the distributions $p_j$ are frequently assumed to be Gaussian, though this is not necessarily the case.

By the extended Kalmanized convex polytope filter update in~\cref{eq:EKCPF}, the posterior \mH-polytopes are defined by,
\begin{equation}\label{eq:Kalmanized-convex-polytope-sum-update}
\begin{aligned}
    P_{i,j}^{X^+_k} &= (I_n - \widetilde{G}_{i,j,k} H_k(\mu_i^{X^-_k}) ) P_i^{X^-_k}\\
    &\qquad + \widetilde{G}_{i,j,k} H_k(\mu_i^{X^-_k})\mu_i^{X^-_k} + K_k y_j^{Y_k} - G_{i,j,k} h_k(\mu_i^{X^-_k}),\\
    G_{i,j,k} &= \Sigma_i^{X^-_k} H_k^T{\left(H_k \Sigma_i^{X^-_k} H_k^T + R_j^{Y_k}\right)}^{-1},\\
    \widetilde{G}_{i,j,k} &= G_{i,j,k}{\left(I_m + \sqrt{I_m + H_k \Sigma_i^{X^-_k} H_k^T \left(R_j^{Y_k}\right)^{-1}}^{-1} \right)}^{-1}
\end{aligned}
\end{equation}
where the gain matrix $G_{i,j,k}$ corresponds to the Kalman gain~\cref{eq:Kalman-filter} and the gain matrix $\widetilde{G}$ corresponds to the modified Kalman gain~\cref{eq:modified-Kalman-gain}.
The resulting approximation to the posterior is given by, 
\begin{equation}
    p_{X^+_k}(x) \approx \sum_{i = 1}^N\sum_{j=1}^M w^{X^+_k}_{i,j} \mathcal{U}(x\,;\, P_{i,j}^{X^+_k}),
\end{equation}
where the weights are given by,
\begin{equation}\label{eq:weights-EnKCPF}
    w^{X^+_k}_{i,j} \propto w^{X^-_k}_i\, w^{Y_k}_j\, \int_{\mathbb{R}^n} \mathcal{U}(x\,;\,P_i^{X^-_k}) \,p_j(h(x)\,;\,y_j^{Y_k},\, R_j^{Y_k}) \mathrm{d}x,
\end{equation}
which, for non-linear measurements~\cref{eq:non-linear-measurement} are non-trivial to compute exactly, but is conducive to a robust approximation for the prior polytopes defined in~\cref{eq:prior-cubes}.

For the cube $Q_{x_{k}^{-,(i)}, h_\U^2 \Sigma^-_k}$ in~\cref{eq:prior-cubes} the $\omega$-point quadrature from~\cref{thm:Omega-point-quadrature-rule} can be used to approximate the integral in~\cref{eq:weights-EnKCPF}, begetting the approximate weights,
\begin{equation}\label{eq:omega-weight-approx}
    w^{X^+_k}_{i,j} \appropto w^{X^-_k}_i\, w^{Y_k}_j\, \sum_{\ell=-n}^n p_j\left(h\left(\Omega_{\ell, x_{k}^{-,(i)}, h_\U^2 \Sigma^-_k}\right)\,;\,y_j^{Y_k},\, R_j^{Y_k}\right),
\end{equation}
where the uniform term is dropped as it is identical in all the components.

The process of approximating the prior using a method such as in~\cref{sec:En-CDE}, performing the Kalmanized convex polytope sum update, and resampling from the posterior as defined in~\cref{sec:polytope-resample} we term the ensemble Kalmanized convex polytope filter (EnKCPF).

For the EnKCPF, instead of using the Chebyshev center, it is computationally cheaper to start the resampling procedure in~\cref{sec:polytope-resample} from the posterior mean of each component, which can be computed using the Kalman filter equations~\cref{eq:Kalman-filter} using the gains defined by~\cref{eq:Kalmanized-convex-polytope-sum-update}.

\begin{corollary}[EnKCPF convergence]
    Observe that $G_{i,j,k}\to0$ as $N\to\infty$, and similarly for $\widetilde{G}_{i,j,k}$. 
    Trivially, this implies that the EnKCPF converges to the BCPF which converges to the BPF which converges in distribution to true Bayesian inference in the limit of ensemble size $N$.
\end{corollary}

\subsection{Static `banana' ensemble filtering example}
\begin{figure}
    \centering
    \includegraphics[width=0.49\linewidth]{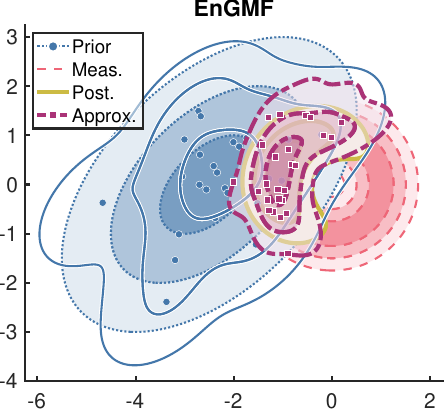}%
    \includegraphics[width=0.49\linewidth]{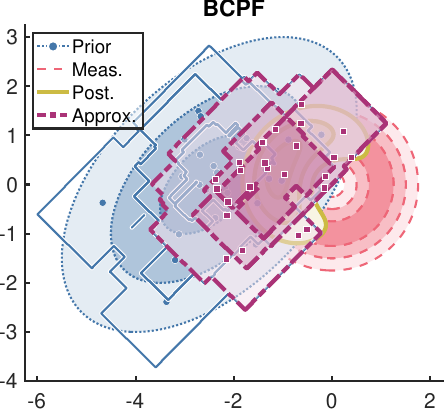}
    \includegraphics[width=0.49\linewidth]{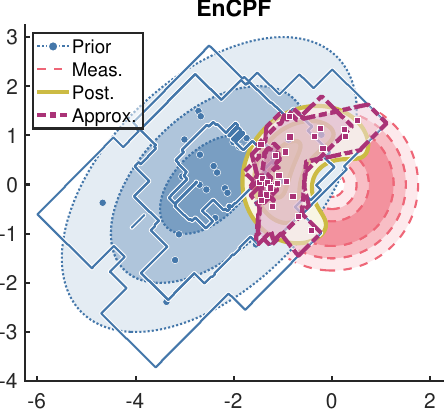}%
    \includegraphics[width=0.49\linewidth]{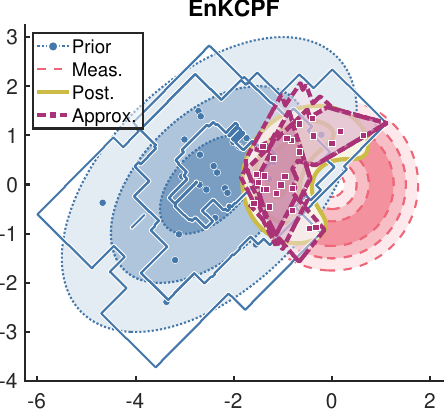}%
    \caption{Evaluation of the EnGMF (top left panel), the BCPF (top right panel), the EnCPF (bottom left panel), and the EnKCPF (bottom right panel) on the static test problems. 
    On all panels, the left blue dotted ovals represent one, two and three standard deviations of prior normal distribution, the solid blue dots represent the prior samples, and the solid blue outlines represent the kernel density estimates using the samples.
    The red dashed outlines represent the measurement distribution with each circle slice representing one, two, and three standard deviations.
    The solid yellow outline resembling a banana represents the exact posterior distribution again with one, two, and three standard deviations.
    The dash-dotted purple outlines represent the respective posterior approximations by their respective algorithms.}
    \label{fig:static-ensemble-example}
\end{figure} 

Similar to~\cref{eq:CPF-example-prior}, take the prior to be Gaussian,
\begin{equation}
    p_{X^-}(x) = \mathcal{N}(x; \mu^-, \Sigma^-),\quad \mu^- = \begin{bmatrix}
        -2.5\\ 0
    \end{bmatrix},\quad \Sigma^- = \begin{bmatrix}
        1 &0.5\\ 0.5 & 1
    \end{bmatrix},
\end{equation}
and take the same measurement functions and distributions from~\cref{eq:EKCPF-example-measurement}. As the posterior defined by this example resembles a banana, it is labeled as such.

We test this example on four different methods: the EnGMF~\cite{anderson1999monte,liu2016efficient,yun2022kernel,popov2024adaptive}, the BCPF defined in~\cref{sec:BCPF-update}, the EnCPF defined in~\cref{sec:EnCPF-update}, and the EnKCPF  defined in ~\cref{sec:EnKCPF-update}.
For the EnCPF the Gaussian measurement error distribution is approximated by the uniform distribution $\U(Q_{1, 1/16})$, with the weights approximated by a grid approach with $10^6$ samples.
The ensemble filters are all given the same identical collection of $N=25$ samples from the prior.
The results of this experiment can be seen in~\cref{fig:static-ensemble-example}.
As can be seen, the EnGMF is overly conservative in determining the posterior, with the BCPF also being very conservative. The EnCPF has a sharp cutoff region around its approximation of the measurement uncertainty distribution, and the EnKCPF has a unique combination of \mH-polytopes that is less conservative, but still looks similar to, the distribution of the EnGMF posterior.

\section{Sequential Filtering Numerical Examples}
\label{sec:numerical-experiments}

The goal of this section is to show that the new \mH-polytope based filter, the EnKCPF can perform adequately in both high and low dimensional settings.
We test several filters, namely the square-root ensemble Kalman filter (EnKF)~\cite{whitaker2002ensemble}, the BPF~\cite{stewart1992use}, the EnGMF as define in~\cite{popov2024adaptive}, the BCPF as defined in~\cref{sec:BCPF-update}, and the EnKCPF~\cref{sec:EnKCPF-update} over two different scenarios.

For all filters except the EnKF we augment the weights with a small defensive factor~\cite{liu2016efficient},
\begin{equation}\label{eq:defensive-factor}
    w_i\,\xleftarrow[]{}\, (1- d_f)\,w_i + d_f/N,
\end{equation}
where we fix $d_f = 10^{-4}$ in order to prevent weight degeneracy.
For all polytope-based filters, the hit-and-run~\cref{sec:polytope-resample} methodology is performed for resampling using $S=25$ steps.

The metric for error that is used in this work is the spatio-temporal root mean squared error (RMSE),
\begin{equation}\label{eq:RMSE}
    \operatorname{RMSE}(x, \mathbb{E}[X^{+}]) = \mathbb{E}_{\operatorname{MC}}\sqrt{\frac{1}{Kn} \sum_{k=1}^K \left\lVert x_{k} - \mathbb{E}[X^+_k]\right\rVert_2^2},
\end{equation}
where $K$ is the number of steps used for computing the error, $n$ is the spatial dimension, $x_k$ is the true state at time index $k$, $\mathbb{E}[X^+_k]$ is our ensemble mean posterior estimate at time index $k$, and the expectation is taken over multiple independent Monte-Carlo realizations.

\subsection{Ikeda Map}
\begin{figure}[t]
    \centering
    \begin{tikzpicture}
    \begin{axis}[clean,
        cycle list name=tol,
        xmode=log,
        log ticks with fixed point,
        xtick={25, 50, 100, 250, 500, 1000, 2500},
        table/col sep=comma,
        xmin = 20,
        xmax = 2600,
        ymin = 0.55,
        ymax = 0.665,
        clip = true,
        xlabel = {Ensemble Size ($N$)},
        ylabel = {Mean Spatio-temporal RMSE},
        every axis plot/.append style={line width=2pt, mark size=3.5pt},
        legend style={at={(1.32,0.7)},anchor=center},
        legend cell align={left}]
    
    \addplot[mark=o,color=tolblue] table [x=N, y=rmseEnGMF, col sep=comma] {data/ikeda.csv};
    \addlegendentry{EnGMF};

    \addplot[mark=square,color=tolred] table [x=N, y=rmseEnKF, col sep=comma] {data/ikeda.csv};
    \addlegendentry{EnKF};

    \addplot[mark=+,color=tolcyan] table [x=N, y=rmseBCPF, col sep=comma] {data/ikeda.csv};
    \addlegendentry{BCPF};

    \addplot[mark=x,color=tolgreen] table [x=N, y=rmseBPF, col sep=comma] {data/ikeda.csv};
    \addlegendentry{BPF};

    \addplot[mark=star,color=tolyellow] table [x=N, y=rmseEnKCPF, col sep=comma] {data/ikeda.csv};
    \addlegendentry{EnKCPF};

    \addplot[mark=none,color=tolpurple, dashed] table [x=N, y=rmseNoFilter, col sep=comma] {data/ikeda.csv};
    \addlegendentry{No Filter};

    \addplot[mark=none,color=tolpurple, dotted] table [x=N, y=rmseBayes, col sep=comma] {data/ikeda.csv};
    \addlegendentry{Bayesian};
    
    \end{axis}
    \end{tikzpicture}
    \caption{Ensemble size ($N$) versus mean spatio-temporal RMSE for the Ikeda map with the EnGMF (blue line with circular markers), EnKF (read line with square markers), BCPF (cyan line with plus markers), BPF (green line with x markers), and the EnKCPF (yellow line with star markers). Additionally, a `no filter' scenario is plotted as the top dashed line, and a true Bayesian inference scenario is plotted as the bottom dotted line.}
    \label{fig:ikeda-experiment}
\end{figure}
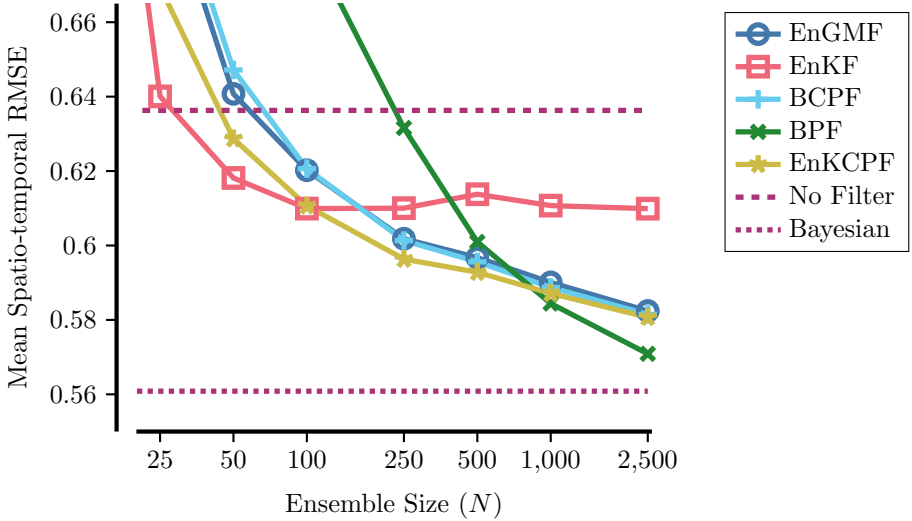

For the first sequential filtering example we make use of the Ikeda map~\cite{ikeda1979multiple, ikeda1980optical, fm2024mapping, osipenko2006dynamical} for its ease of implementation and speed.
The Ikeda map is a two dimensional discrete dynamical system that is the composition of three distinct actions: rotation by an angle derived from the inputs, a scaling by a constant, and a shift of one of the variables.
The particular version of the Ikeda map that is used for this work is given by,
\begin{equation}\label{eq:Ikeda-map}
\begin{gathered}
    \begin{bmatrix}
        v_{k+1}\\
        w_{k+1}
    \end{bmatrix}
    = \begin{bmatrix}
    1 + 0.9 \left(v_k \cos \theta_k - w_k\sin\theta_k\right)
    \\
    0.9 \left(v_k \sin\theta_k + w_k\cos\theta_k\right)
    \end{bmatrix},\\
    \text{with}\qquad \theta_k = 0.4 - \frac{6}{1 + v_k^2 + w_k^2},
    \end{gathered}
\end{equation}
where $v$ is the first variable, $w$ is the second variable, and~\cref{eq:Ikeda-map} propagates the system from time index $k$ to time index $k+1$.
We assume that we know the dynamics exactly and that their are no sources of model error.

For the measurement we again assume a range measurement,
\begin{equation}
    h(x) = \lVert x \lVert_2, \,\, H(x) = \frac{x^T}{\lVert x \lVert_2},\quad p_{Y_k}(y) = \mathcal{N}(y; h(x_k), 1),
\end{equation}
with Gaussian measurement error centered on the true range with variance one.

For the BPF we apply Gaussian process noise before the update step with covariance of $10^{-8} I_2$, as without it, the BPF does not converge.
The EnKF is tuned with an inflation factor~\cite{asch2016data} of $1.001$ in order to prevent particle collapse.
It is worth noting that the EnKF that we make use of computes the covariance matrices in a linearized and not statistical manner which has been shown to have superior performance for low-dimensional problems (and incidentally make the EnKF a non-linear filter) in~\cite{michaelson2023ensemble}.
We additionally run a particle filter~\cite{reich2015probabilistic} with an excess of  particles to create a baseline theoretically minimum error which we term `Bayesian inference', and run a `no filter' scenario with the same number of particles to get a theoretical maximum error if no filtering is performed.

We run all filters for 550 steps discarding the first 50 over 1008 Monte Carlo samples and compute the mean spatio-temporal RMSE~\cref{eq:RMSE} varying the ensemble size from $N=25$ to $N=2500$. 
The results of this experiment can be seen in~\cref{fig:ikeda-experiment}.
It is interesting to see that the BCPF performs in a similar manner to that of the EnGMF even though both filters are built on widely different architectures, only sharing a common idea of using kernel density estimation, and a common hidden Gaussian assumption on the optimality of the bandwidth parameter.
The EnKF starts off as the best filter for small ensemble sizes, likely due to it requiring less information for good filter performance.
The EnKCPF then briefly overtakes the EnKF for $N=250$ and $N=500$ with the BPF overtaking the KDE-based filters for all larger ensemble sizes, as the KDE-based methods likely create conservative estimates of the underlying probability distribution.
The EnGMF, BCPF, and EnKCPF can likely have superior performance with optimal filter tuning, though this is not the focus of this work.
The EnKCPF is always the first or second best filter out of all the ones presented for this scenario likely as the \mH-polytope mixture assumption can likely better approximate the compact attractor of the Ikeda map than the EnGMF, and the Kalmanization of the filter makes better use of the available limited information than the BCPF.

\subsection{Lorenz '96}
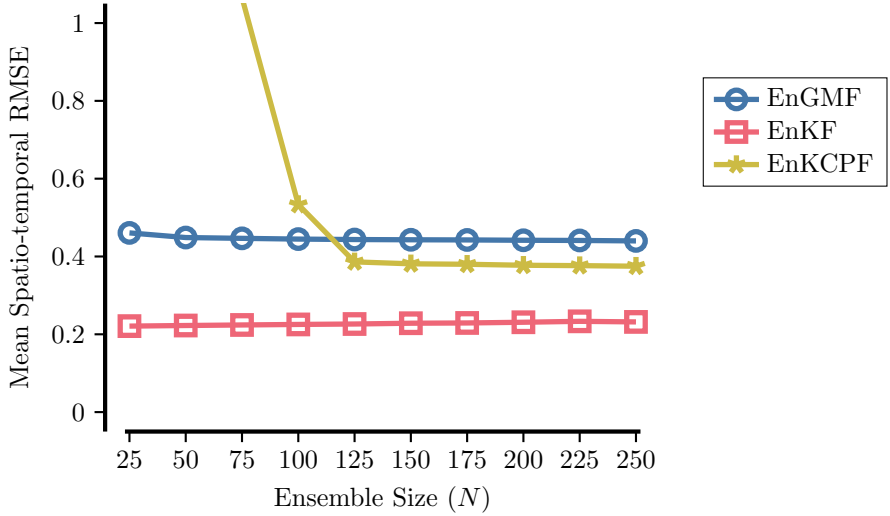
\begin{figure}[t]
    \centering
    \begin{tikzpicture}
    \begin{axis}[clean,
        cycle list name=tol,
        xtick={25, 50, 75, 100, 125, 150, 175, 200, 225, 250 },
        table/col sep=comma,
        xmin = 23,
        xmax = 252,
        ymin = -0.05,
        ymax = 1.05,
        clip = true,
        xlabel = {Ensemble Size ($N$)},
        ylabel = {Mean Spatio-temporal RMSE},
        every axis plot/.append style={line width=2pt, mark size=3.5pt},
        legend style={at={(1.30,0.7)},anchor=center},
        legend cell align={left}]
    
    \addplot[mark=o,color=tolblue] table [x=N, y=rmseEnGMF, col sep=comma] {data/l96.csv};
    \addlegendentry{EnGMF};

     \addplot[mark=square,color=tolred] table [x=N, y=rmseEnKF, col sep=comma] {data/l96.csv};
    \addlegendentry{EnKF};

    \addplot[mark=star,color=tolyellow] table [x=N, y=rmseEnKCPF, col sep=comma] {data/l96.csv};
    \addlegendentry{EnKCPF};

    
    \end{axis}
    \end{tikzpicture}
    \caption{Ensemble size ($N$) versus mean spatio-temporal RMSE for the Lorenz '96 equations with the EnGMF (blue line with circular markers), EnKF (red line with square markers), and the EnKCPF (yellow line with star markers). }
    \label{fig:Lorenz-96-experiment}
\end{figure}

The goal of the second numerical experiment is to show that the EnKCPF has a better error constant than the EnGMF in the high-dimensional setting. 
We make use of the standard $40$-variable Lorenz '96 equations~\cite{lorenz1996predictability,van2018dynamics},
\begin{equation}
    x_k' = -x_{k-1}(x_{k-2} - x_{k+1}) - x_k + F,\quad k = 1,\dots, 40,
\end{equation}
where by the cyclic boundary conditions, $x_0 = x_{40}$, $x_{-1} = x_{39}$, and $x_{41} = x_{1}$. The forcing is set to $F=8$ to have a chaotic system with a Kaplan-Yorke dimension of $27.1$ and $13$ positive Lyapunov exponents~\cite{popov2019bayesian}.
At each time index, the system is propagated with one step of the canonical fourth order explicit Runge-Kutta method with a time step of $\Delta t = 0.05$.

For the measurement, we take the simple linear measurements,
\begin{equation}
    h(x) = x,\quad H(x) = I_{40},\quad R = I_{40},
\end{equation}
observing all variables independently with a measurement error of identity.

In order to estimate the covariance using a small ensemble we make use of B-localization~\cite{asch2016data}, thus for all the filters tested a Gaussian decorrelation function with radius $r=3$ is applied to the statistical covariance of the ensemble in order for the covariance estimate to not have an influence on the performance of the filters.
The BPF and BCPF are not run for this scenario as they both fail to converge. 
Additionally a no filter scenario is not plotted as the error is significantly higher than the error of the filters.
No good estimate of the lowest possible error exists for this scenario as no convergent filter is known to have a high enough rate of convergence.

We run all three filters for 1100 steps discarding the first 100 over 108 Monte Carlo samples and compute the mean spatio-temporal RMSE~\cref{eq:RMSE} varying the ensemble size from $N=25$ to $N=250$. 
The results of this experiment can be seen in~\cref{fig:Lorenz-96-experiment}.
While the EnKF is the best filter in this scenario, it could not perform as well as the convergent filters for high ensemble sizes for the Ikeda map.
This means that out of all tested filters, the EnGMF and the EnKCPF are the only convergent filters that work for a difficult low-dimensional setting, and an easy (for the EnKF) high-dimensional setting.

While the EnKF and the EnGMF both do not diverge in the low ensemble size regime, the EnKCPF does not.
The EnKCPF however does attain lower error than the EnGMF for a larger ensemble size.
The poor performance of the EnKCPF for smaller ensemble members is likely due to the approximation of the prior in~\cref{eq:prior-cubes}.
The unit $n$-cube~\cref{eq:unit-cube} can embed a sphere of radius $1/2$, but can be circumscribed by a sphere of radius $\sqrt{n}/2$, meaning that it likely is a poor approximation for higher dimensions.

\section{Conclusions and future outlook}
\label{sec:conclusions}

In this work we have presented a geometric look at Bayesian inference based on (convex) \mH-polytopes.
We have developed the convex polytope filter (CPF), the extended convex polytope filter (ECPF), the Kalmanized convex polytope filter (KCPF), its extended counterpart (EKCPF), and three convergent ensemble based filters: the bootstrap convex polytope filter (BCPF), the ensemble convex polytope filter (EnCPF), and finally the ensemble Kalmanized convex polytope filter (EnKCPF).
We have shown through numerical examples that this methodology holds promise for both low dimensional and high dimensional applications.

The largest limiting factor of the methodology developed herein is the simplicity of the prior kernel density estimate~\cref{eq:prior-cubes}.
This estimate does allow for known means and covariances to be used for the EnKCPF, making it simple to implement and computationally tractable, but it suffers in the high-dimensional setting from a large discrepancy between the distance from the center to the facets and from the center to the vertices.

There are a lot of potential future directions that could be of significant utility to this existing methodology, such as, 
better approximations to the EnCPF measurements and weights could elevate it to a useful algorithm,
a better approximation to the prior polytope mixture from the ensemble, with good approximations to the mean and covariance,
potentially weakening the convergent filter requirement and adding methodologies such as R-localization~\cite{asch2016data},
and making use of preexisting linear constraints such as in~\cite{gupta2007kalman, prakash2010constrained}.

\section*{Acknowledgments}
The first author would like to acknowledge Evie Nielen for conversations on polytopes that tangentially inspired this work.

\bibliographystyle{siamplain}
\bibliography{biblio}

\end{document}